%% file: SGC.tex
\documentclass[leqno]{amsart}
\usepackage[a4paper, margin=3cm]{geometry}
\usepackage{amsthm}
\usepackage{amsmath,amssymb,mathtools,mathdots}
\usepackage{mathrsfs}
\usepackage{url}
\usepackage[all]{xy}
\usepackage{comment}
\usepackage{tikz}
\usepackage{pgfplots}
\pgfplotsset{compat=1.18}
\newtheorem{theorem}{Theorem}[section]
\newtheorem*{theorem*}{Theorem}
\newtheorem{lemma}[theorem]{Lemma}
\newtheorem*{lemma*}{Lemma}

\newtheorem*{proposition*}{Proposition}

\newtheorem*{corollary*}{Corollary}

\newtheorem*{claim*}{Claim}

\newtheorem*{fact*}{Fact}

\newtheorem*{conjecture*}{Conjecture}
\theoremstyle{definition}
\newtheorem{definition}[theorem]{Definition}
\newtheorem*{definition*}{Definition}

\newtheorem*{example*}{Example}
\newtheorem{remark}[theorem]{Remark}
\newtheorem*{remark*}{Remark}

\newtheorem*{question*}{Question}

\newtheorem*{assumption*}{Assumption}
\numberwithin{equation}{section}
\allowdisplaybreaks[1]

\newcommand{\C}{{\mathbb C}}
\newcommand{\R}{{\mathbb R}}

\newcommand{\Z}{{\mathbb Z}}
\newcommand{\N}{{\mathbb N}}

\title[A zero-one law for spectral gap in $\mathrm{SU}(2)$]{A zero-one law for the spectral gap\\property of pairs in $\mathrm{SU}(2)$}

\author{Oleg Pikhurko}
\author{Kohki Sakamoto}

\thanks{O.~Pikhurko: Mathematics Institute and DIMAP, University of Warwick, Coventry CV4 7AL, UK. \\ \indent K.~Sakamoto: Graduate School of Mathematical Sciences, the University of Tokyo, 3-8-1 Komaba, Meguro, Tokyo 153-8914, Japan. \\ \indent Corresponding author: Kohki Sakamoto (\texttt{sakamoto-kohki571@g.ecc.u-tokyo.ac.jp})}

\date{} 

\begin{document}

\begin{abstract}
For $n \ge 2$, Gamburd, Jakobson, and Sarnak [\textit{J. Eur. Math. Soc.} 1, 51–85 (1999)] conjectured that almost every $n$-tuple in $\mathrm{SU}(2)$ has a spectral gap. Toward this conjecture, Fisher [\textit{Int. Math. Res. Not.} (2006), Art. ID 26028] established a zero-one law for $n \ge 3$, but obtained only a partial result for $n=2$. In this paper, we prove that the zero-one law also holds for $n=2$. We also remark that a Baire categorical analogue of this result holds.
\end{abstract}
\maketitle

\section{Introduction}
Let $n \ge 2$ and consider $n$-tuples in the special unitary group $\mathrm{SU}(2)$. The \textit{spectral gap conjecture} of Gamburd, Jakobson, and Sarnak \cite{MR1677685} (see also Lubotzky, Phillips, and Sarnak \cite{MR861487}) asserts that almost every $n$-tuple (with respect to the Haar measure) generates a subgroup whose natural action on $L^2(\mathrm{SU}(2))$ has a spectral gap (see Definition \ref{def:sg}). \\
\indent The study of spectral gaps for tuples in $\mathrm{SU}(2)$ originates from the Banach–Ruziewicz problem, which asks whether the uniform measure is the unique rotationally invariant finitely additive probability measure defined on all Lebesgue measurable subsets of $S^{2}$. Rosenblatt showed that it suffices to find a finite subset of the rotation group whose action on $L^{2}(S^{2})$ has a spectral gap \cite[Theorem~1.4]{MR610970}. Then, Drinfeld \cite{MR757256} constructed such subsets using heavy machinery
from the theory of automorphic forms. See also Margulis \cite{MR596890} and
Sullivan \cite{MR590825} for earlier results in higher dimensions. Further, Gamburd, Jakobson, and Sarnak \cite{MR1677685} constructed $n$-tuples in $\mathrm{SU}(2)$ with a spectral gap based on just elementary arguments using quaternions. In their influential paper, Bourgain and Gamburd \cite{MR2358056} extended this to $n$-tuples with algebraic entries that generate a dense subgroup of $\mathrm{SU}(2)$, but the set of such tuples has zero measure. \\
\indent As a first step toward resolving this conjecture, Fisher \cite{MR2250018} proved a zero-one law for spectral gap for $n \ge 3$, namely, that the set of $n$-tuples with spectral gap has zero measure or full measure. His proof is based on the ergodicity of the action of $\mathrm{Out}(F_{n})$, the outer automorphism group of the free group on $n$ generators, on the character variety $\mathrm{Hom}(F_{n}, \mathrm{SU}(2)) / \mathrm{SU}(2)$, which was established by Goldman in \cite{MR2346275}. However, when $n=2$, it is known that the action is no longer ergodic, and indeed has many ergodic components described by the trace of the commutator as shown by Pickrell and Xia \cite[Theorem~2.1.4]{MR1915045}. Therefore, Fisher stated the corresponding partial result for $n=2$ (see Theorem \ref{Fisher2}).\\
\indent In this paper, we point out that the zero-one law also holds for $n=2$. Let $\mathbf{SG}_{2} \subset \mathrm{SU}(2)^{2}
$ denote the set of pairs having spectral gap.
\begin{theorem}\label{mainthm}
    The set of pairs with spectral gap $\mathbf{SG}_{2} \subset \mathrm{SU}(2)^{2}$ has either zero measure or full measure with respect to the Haar measure. 
\end{theorem}
Our argument is elementary and relies on the observation that the set of non-spectral-gap pairs is stable under certain word maps, allowing the locus to move across ergodic components. Further, we also remark that a Baire categorical analogue of the zero-one law holds.
\begin{theorem}\label{baire}
    The set of pairs with spectral gap $\mathbf{SG}_{2} \subset \mathrm{SU}(2)^{2}$ is either meager or comeager.
\end{theorem}

\subsection*{Acknowledgements}
We would like to thank Andr\'as M\'ath\'e for his helpful discussions. We also thank the anonymous referee for helpful comments and suggestions, in particular for suggesting a simpler proof of Lemma \ref{planezeroone} and providing the code for Figure \ref{fig:fibers}. OP is supported by ERC Advanced Grant 101020255. KS is supported by JSPS KAKENHI Grant number 25KJ0862 and FoPM, WINGS Program, the University of Tokyo. KS also thanks the Mathematics Institute at the University of Warwick for their hospitality during his visit, where this work was carried out.

\section{Preliminaries}
\indent Let $K$ denote the special unitary group $\mathrm{SU}(2)$. First, we recall the definition of spectral gap for $n$-tuples in $K$. The following definition differs in form from that in \cite{MR1677685}, but the two are equivalent (see \cite{MR2250018} for details). 
\begin{definition}\label{def:sg}
 Let $\lambda$ denote the left regular representation of $K$ on $L^{2}(K)$, and let $L^{2}_{0}(K)$ denote the orthogonal complement of the constant functions. An $n$-tuple $(g_{1},\ldots,g_{n}) \in K^{n}$ is said to \textit{have spectral gap} if the restriction of $\lambda$ to $\Gamma = \langle g_{1},\ldots,g_{n} \rangle$ has spectral gap, i.e., there exists $\varepsilon >0$ such that for every element $v \in L_{0}^{2}(K)$ there exists $k\in\{1,\ldots,n\}$ such that $\left\lVert \lambda(g_{k})v-v \right\rVert \ge \varepsilon \left\lVert v \right\rVert$. Let $\mathbf{SG}_{n}$ denote the set of $n$-tuples having spectral gap.
\end{definition}
It is easy to see that the set $\mathbf{SG}_{n}$ is Borel. As Fisher explains in \cite{MR2250018}, this property actually depends only on the generated subgroup, and this is the starting point of his argument. Namely, it means that the property is invariant under the natural action of the automorphism group $\mathrm{Aut}(F_{n})$ on $K^{n}$, and of course it is also invariant under simultaneous conjugation, and hence the image of $\mathbf{SG}_{n}$ in the character variety $\mathrm{Hom}(F_{n}, K)/K$ is invariant under the natural action by $\mathrm{Out}(F_{n})$. Since this action is ergodic for $n\ge3$ (\cite{MR2346275}), the zero-one law follows. Fisher also obtained a partial result for $n=2$ based on the explicit description of the ergodic components in \cite{MR1915045}. Let us recall his result here.

\begin{theorem}[Fisher, Theorem 2.3 in \cite{MR2250018}]\label{Fisher2}
    Let $g \colon K^{2} \to [-2,2]$ denote the function defined by $g(a,b) = \mathrm{tr}([a,b])$, the trace of the commutator $aba^{-1}b^{-1}$. Then, for almost every $t \in [-2,2]$, the set $g^{-1}(t) \cap \mathbf{SG}_{2}$ is null or conull with respect to the disintegrated measure on the fiber $g^{-1}(t)$.
\end{theorem}

Therefore, to establish the zero-one law for $n=2$, we need some additional maps that preserve the spectral gap property and mix fibers of the function $g$. As we will see later, this can be done in a surprisingly simple way. The key is the following.

\begin{lemma}\label{stable}
    If $(a,b) \notin \mathbf{SG}_{2}$, then $(x,y) \notin \mathbf{SG}_{2}$ for every $x,y \in \langle a,b \rangle$.
\end{lemma}

\begin{proof}
Let $\ell(w)$ denote the word length of $w\in\langle a,b\rangle$ with respect to $\{a,b\}$. Since $\lambda$ is unitary, $\|\lambda(s^{-1})v-v\|=\|\lambda(s)v-v\|$ for $s\in\{a,b\}$. The triangle inequality and induction on $\ell(w)$ give
\[
\|\lambda(w)v-v\|\le\ell(w)\max\{\|\lambda(a)v-v\|,\|\lambda(b)v-v\|\}
\]
for every $v\in L^{2}(K)$. Thus any sequence of almost invariant unit vectors in $L^{2}_{0}(K)$ for $(a,b)$ is also almost invariant for each fixed pair $(x,y)\in\langle a,b\rangle^{2}$.
\end{proof}

\section{Proof of the zero-one law for $n=2$}
In this section, we establish the zero-one law for $\mathbf{SG}_{2}$ using Theorem \ref{Fisher2} and Lemma \ref{stable}. In fact, it is enough to consider the map $\Phi \colon (a,b) \mapsto (a^{2},b)$ for our purpose. First, let us observe how this map affects the value of $g$. Recall that the function $g \colon K^{2} \to [-2,2]$ is given by the trace of the commutator.

\begin{lemma}\label{trcomm}
    For $(a,b) \in K^{2}$, let $x$ and $t$ denote $\mathrm{tr}(a)$ and $\mathrm{tr}([a,b])$, respectively. Then, it holds that
    \begin{itemize}
    \item $\mathrm{tr}(a^{2})=x^{2}-2$,
    \item $\mathrm{tr}([a^{2},b]) = x^{2}(t-2)+2$.
    \end{itemize}
\end{lemma}

\begin{proof}
By the Cayley--Hamilton theorem, we have $a^{2} = xa-I$ and $a^{-2} = xa^{-1} - I$. The first formula in the lemma follows immediately. Further, by $\mathrm{tr}(a) = \mathrm{tr}(a^{-1})$ we have 
\begin{align*}
    \mathrm{tr}([a^{2},b]) &= \mathrm{tr}(a^{2}ba^{-2}b^{-1}) = x^{2}\,\mathrm{tr}([a,b]) - x\,\mathrm{tr}(a) -x\,\mathrm{tr}(ba^{-1}b^{-1}) + \mathrm{tr}(I)=x^{2}(t-2) + 2
\end{align*}
as desired.
\end{proof}

By Lemma \ref{trcomm}, the map $\Phi$ induces the map $\phi(x,t)=(x^{2}-2,x^{2}(t-2)+2)$ in the coordinates $(x,t)=(\mathrm{tr}(a),\mathrm{tr}([a,b]))$. Since the complement of $\mathbf{SG}_{2}$ is stable under $\Phi$ by Lemma \ref{stable}, we study the dynamics of $\phi$. First, we determine the domain of these coordinates and the corresponding pushforward measure.

\begin{lemma}\label{locus}
    Consider the map $\Pi \colon K^{2} \to [-2,2]^{2}$ defined by $\Pi(a,b) = (\mathrm{tr}(a), \mathrm{tr}([a,b]))$. Then, the image is given by 
    \[
    \Pi(K^{2}) = \{(x,t) \in [-2,2]^{2} \colon x^{2}-2 \le t \},
    \]
    which will be denoted by $D$. Further, the pushforward of the Haar measure on $K^{2}$ is equivalent to the Lebesgue measure on $D$.
\end{lemma}

\begin{proof}
Every element of $K$ is conjugate to $\operatorname{diag}(e^{i\alpha},e^{-i\alpha})$ for some $\alpha\in[0,\pi]$, so its trace belongs to $[-2,2]$. Conjugating $a$ and $b$ by the same element of $K$, we may assume that
\[
a=\begin{pmatrix}e^{i\alpha}&0\\0&e^{-i\alpha}\end{pmatrix},
\quad
b=\begin{pmatrix}u&-\overline v\\v&\overline u\end{pmatrix},
\quad |u|^{2}+|v|^{2}=1.
\]
Put $x=2\cos\alpha$ and $s=|v|^{2}$. A direct computation, using $\cos(2\alpha)=x^{2}/2-1$, gives
\[
t=\operatorname{tr}([a,b])=2|u|^{2}+2|v|^{2}\cos(2\alpha)=2-s(4-x^{2}).
\]
Since $0\le s\le1$, this implies $x^{2}-2\le t\le2$. Conversely, for any $(x,t)\in D$ with $|x|<2$, choose $\alpha$ with $2\cos\alpha=x$, set $s=(2-t)/(4-x^{2})$, and take $u=\sqrt{1-s}$ and $v=\sqrt{s}$. This realizes $(x,t)$. If $|x|=2$, then $t=2$, and we may take $a=\pm I$ and $b=I$. Hence $\Pi(K^{2})=D$.

Let $\mu$ denote the normalized Haar measure on $K^{2}$. Since the conjugation above may be chosen to depend only on $a$, and $b$ is independent of $a$, the conditional law of the conjugated $b$ given $a$ remains the Haar measure on $K$. Under the standard identification $K\simeq S^{3}\subset\C^{2}$, the Haar measure on $K$ corresponds to the uniform measure on $S^{3}$. In the coordinates
\[
(u,v)=\bigl(\sqrt{1-s}\,e^{i\theta},\sqrt{s}\,e^{i\psi}\bigr),
\quad 0<s<1,\quad 0\le\theta,\psi<2\pi,
\]
this measure is $(4\pi^{2})^{-1}\,ds\,d\theta\,d\psi$, so $s$ is uniform on $[0,1]$. Also, since $x/2$ is one real coordinate of a uniform point on $S^{3}$, $x$ has density $(2\pi)^{-1}\sqrt{4-x^{2}}$ on $[-2,2]$.

Thus, for $|x|<2$, the conditional distribution of $t$ given $a$ is uniform on $[x^{2}-2,2]$. Consequently, $\Pi_{*}\mu$ gives no mass to the boundary $\partial D$, and on the interior $D^{\circ}$ we have
\[
d(\Pi_{*}\mu)(x,t)=\frac{1}{2\pi\sqrt{4-x^{2}}}\,dx\,dt,
\]
which completes the proof.
\end{proof}

Thus $\phi$ maps $D$ into itself and $\Pi\circ\Phi=\phi\circ\Pi$. The main theorem will follow from the next lemma. 

\begin{lemma}\label{planezeroone}
 Let $p\colon D \to [-2,2]$ denote the projection $(x,t) \mapsto t$. Suppose that a Borel subset $E \subset D$ satisfies the following:
 \begin{itemize}
     \item $E$ is stable under $\phi$ modulo null sets, i.e., $E\setminus\phi^{-1}(E)$ is null.
     \item For almost every $t\in[-2,2]$, the fiber $p^{-1}(t)$ is either contained in $E$ or in its complement modulo null sets.
 \end{itemize}
 Then $E$ is either null or conull in $D$.
\end{lemma}

\begin{proof}[Proof of Theorem \ref{mainthm}]
Let $\mu$ be the normalized Haar measure on $K^{2}$ and let $N=K^{2}\setminus\mathbf{SG}_{2}$. By Theorem \ref{Fisher2}, there is a Borel set $I\subset[-2,2]$ such that $N=g^{-1}(I)$ modulo $\mu$-null sets. Set $E=p^{-1}(I)\subset D$. Since $g=p\circ\Pi$, we have $N=\Pi^{-1}(E)$ modulo $\mu$-null sets.

The map $\Phi(a,b)=(a^{2},b)$ is nonsingular with respect to $\mu$. Indeed, if the eigenvalues of $a$ are $e^{\pm i\alpha}$, a direct computation shows that the Jacobian determinant of the map $a\mapsto a^{2}$ is
\[
8\cos^{2}\alpha=2\operatorname{tr}(a)^{2}.
\]
Thus it is a local diffeomorphism whenever $\operatorname{tr}(a)\ne0$. Since the set $\{a\in K:\operatorname{tr}(a)=0\}$ has Haar measure zero, it follows that $\Phi_{*}\mu\ll\mu$.

By Lemma~\ref{stable}, $N\subset\Phi^{-1}(N)$. Since $\Phi$ is nonsingular and $N=\Pi^{-1}(E)$ modulo null sets, the identity $\Pi\circ\Phi=\phi\circ\Pi$ gives
\[
\mathbf{1}_E\circ\Pi\leq\mathbf{1}_E\circ\phi\circ\Pi
\quad \mu\text{-a.e.}
\]
Since $\Pi_*\mu$ is equivalent to Lebesgue measure on $D$ by Lemma~\ref{locus}, it follows that $E\setminus\phi^{-1}(E)$ is null. The second assumption of Lemma~\ref{planezeroone} holds by the definition of $E$. Therefore Lemma~\ref{planezeroone} applies, and $E$ is either null or conull. Since $\mu(N)=(\Pi_*\mu)(E)$, this proves the theorem.
\end{proof}

\begin{proof}[Proof of Lemma \ref{planezeroone}]
For $c\in[-2,2]$, $\phi$ maps the horizontal fiber $H_c=p^{-1}(c)$ to the line segment joining $(-2,2)$ and $(c,c^{2}-2)$. Its projection onto the $t$-axis is $[c^{2}-2,2]$ (see Figure \ref{fig:fibers}). To prove the lemma, we use the following coordinates. Let $\mathbb T=\R/(2\pi\Z)$, and equip $\mathbb T\times[0,1]$ with normalized product Lebesgue measure $\nu$. Define
\[
T(\alpha,r)=(2\alpha,r),
\quad
f(\alpha,r)=\bigl(2\cos\alpha,\,2(1-r)\cos(2\alpha)+2r\bigr).
\]
For $x=2\cos\alpha$, the second coordinate of $f$ is $(1-r)(x^{2}-2)+2r$, so $f$ takes values in $D$. A direct calculation gives $f\circ T=\phi\circ f$.
The restrictions of $f$ to $(0,\pi)\times(0,1)$ and $(\pi,2\pi)\times(0,1)$ are diffeomorphisms onto $D^{\circ}$. The union of these two domains has full measure, so $f_{*}\nu$ is equivalent to Lebesgue measure on $D$.

Let $F=f^{-1}(E)$. The first assumption and the relation $f\circ T=\phi\circ f$ give $F\subset T^{-1}(F)$ modulo $\nu$-null sets. Since $T$ preserves $\nu$, this inclusion is an equality modulo null sets. By Fubini's theorem and the ergodicity of the doubling map on the circle, for almost every $r\in[0,1]$ the set $\{\alpha:(\alpha,r)\in F\}$ is null or conull in $\mathbb T$.

By the second assumption, there is a Borel set $I\subset[-2,2]$ such that $E=p^{-1}(I)$ modulo null sets. Hence
\[
\mathbf1_F(\alpha,r)=\mathbf1_I\bigl(2(1-r)\cos(2\alpha)+2r\bigr)
\quad\nu\text{-a.e.}
\]
For every $r\in[0,1)$, the pushforward of Lebesgue measure by $\alpha\mapsto2(1-r)\cos(2\alpha)+2r$ is equivalent to Lebesgue measure on $(4r-2,2)$. It follows that, for almost every $r\in(0,1)$, the set $I$ is null or conull in this interval. Any two such intervals overlap in a set of positive measure, so either $I$ is null or conull in every such interval. Since such values of $r$ occur arbitrarily close to $0$, the union of these intervals is $(-2,2)$, and hence $I$, and therefore $E$, is null or conull.
\end{proof}

\begin{figure}[!htb]
\centering
\input{figures/fiber_geometry.tex}
\caption{The fiber $H_c=p^{-1}(c)$ and its image $\phi(H_c)$, shown in red for $c=1/4$. The blue curve is $t=x^{2}-2$. The green dashed lines $t=x$ and $x=c$ meet at $(c,c)$, and the vertical line meets the blue curve at the endpoint $(c,c^{2}-2)$ of $\phi(H_c)$.}
\label{fig:fibers}
\end{figure}

\section{Baire categorical version}
The topological version (Theorem \ref{baire}) also follows from Lemma \ref{stable}.

\begin{proof}[Proof of Theorem \ref{baire}]
    First, suppose that $\mathbf{SG}_{2}$ contains every pair generating a dense subgroup of $K$. Then, since the set of pairs generating dense subgroups is comeager in $K^{2}$, it follows that $\mathbf{SG}_{2}$ is comeager.\\
    \indent Second, suppose that there exists $(a,b) \notin \mathbf{SG}_{2}$ generating a dense subgroup of $K$. We will show that $\mathbf{SG}_{2}$ is meager in this case, by writing $K^{2}\setminus\mathbf{SG}_{2}$ as a countable intersection of dense open sets. By the definition of the spectral gap property, we have
    \[
    K^{2}\setminus \mathbf{SG}_{2} = \bigcap_{n\in \N} \left\{(a,b) \in K^{2} \colon \exists v \in L_{0}^{2}(K), \left\lVert \lambda(a)v-v\right\rVert + \left\lVert \lambda(b)v-v  \right\rVert < \frac{\left\lVert v\right\rVert}{n} \right\}. 
    \]
    Let $(X_{n})_{n \in \N}$ denote the sequence of subsets appearing on the right-hand side. Since the left regular representation is strongly continuous, it follows that each $X_{n}$ is open. Further, Lemma \ref{stable} implies that $K^{2}\setminus \mathbf{SG}_{2}$ contains $\langle a,b\rangle^{2}$. Since $\langle a,b\rangle$ is dense in $K$, this product is dense in $K^{2}$, and hence each $X_{n}$ is also dense. This completes the proof.
\end{proof}

\begin{remark}\label{rem:andras}
    For every integer $n\ge2$, the set $\mathbf{SG}_n\subset K^n$ of $n$-tuples with spectral gap is either meager or comeager. Clearly, if $\mathbf{SG}_2$ is comeager, then $\mathbf{SG}_n$ ($n \ge 3$) is comeager since adding generators preserves the spectral gap. On the other hand, as pointed out by Andr\'as M\'ath\'e, if $\mathbf{SG}_2$ is meager then $\mathbf{SG}_n$ is meager for every $n \ge 3$. Indeed, by the proof of Theorem \ref{baire}, there exists a pair $(a, b) \notin \mathbf{SG}_2$ generating a dense subgroup $\Gamma \subset \mathrm{SU}(2)$. For any $n \ge 2$, the product $\Gamma^n$ is dense in $\mathrm{SU}(2)^n$ and every $n$-tuple in $\Gamma^n$ generates a subgroup lacking a spectral gap. Since $\mathrm{SU}(2)^n \setminus \mathbf{SG}_n$ is also a $G_\delta$ set, it is comeager, which implies $\mathbf{SG}_n$ is meager for all $n \ge 2$.
\end{remark}

\bibliographystyle{amsplain}
\bibliography{SGC}

\end{document}

%% file: figures/fiber_geometry.tex
\begin{tikzpicture}
\pgfmathsetmacro{\fiberc}{0.25}
\pgfmathsetmacro{\fiberhalfwidth}{sqrt(\fiberc+2)}
\pgfmathsetmacro{\imageheight}{\fiberc*\fiberc-2}
\begin{axis}[
    width=9cm,
    height=9cm,
    axis lines=middle,
    xmin=-2.15, xmax=2.15,
    ymin=-2.15, ymax=2.15,
    xtick={-2,-1,0,1,2},
    ytick={-2,-1,0,1,2},
    yticklabels={$-2$,$-1$,,{$1$},{$2$}},
    xlabel={$x$},
    ylabel={$t$},
    xlabel style={at={(axis description cs:1.02,0.5)},anchor=west},
    ylabel style={at={(axis description cs:0.5,1.02)},anchor=south},
    ticklabel style={font=\scriptsize},
    label style={font=\small},
    grid=major,
    grid style={gray!20,very thin},
    axis line style={gray!75,thin},
    tick style={gray!75,thin},
    clip=false,
]
\addplot[blue!75!black,thick,domain=-2:2,samples=201] {x^2-2};
\addplot[blue!75!black,thin] coordinates {(-2,2) (2,2)};

\addplot[green!45!black,dashed,thick,domain=-2:2,samples=2] {x};
\addplot[green!45!black,dashed,thick]
    coordinates {(\fiberc,\imageheight) (\fiberc,2)};

\addplot[red!80!black,thick]
    coordinates {(-\fiberhalfwidth,\fiberc) (\fiberhalfwidth,\fiberc)};
\addplot[red!80!black,thick]
    coordinates {(-2,2) (\fiberc,\imageheight)};

\addplot[only marks,mark=*,mark size=1.4pt,red!80!black]
    coordinates {(-2,2) (\fiberc,\imageheight)};
\addplot[only marks,mark=*,mark size=1.4pt,green!45!black]
    coordinates {(\fiberc,\fiberc)};

\node[font=\small,anchor=south,text=red!80!black]
    at (axis cs:1.03,0.30) {$H_c$};
\node[font=\small,anchor=west,text=red!80!black]
    at (axis cs:-1.37,1.05) {$\phi(H_c)$};
\node[font=\scriptsize,anchor=north west,text=green!45!black]
    at (axis cs:-1.80,-1.84) {$t=x$};
\node[font=\scriptsize,anchor=east,text=green!45!black]
    at (axis cs:-0.16,0.52) {$(c,c)$};
\node[font=\scriptsize,anchor=west,text=green!45!black,
      fill=white,inner sep=1pt,xshift=3pt]
    at (axis cs:\fiberc,1.44) {$x=c$};
\draw[gray!65,very thin]
    (axis cs:\fiberc,\imageheight) -- (axis cs:0.96,-1.80);
\node[font=\scriptsize,anchor=west,text=red!80!black]
    at (axis cs:0.99,-1.80) {$(c,c^{2}-2)$};
\node[font=\small] at (axis cs:1.03,1.56) {$D$};
\end{axis}
\end{tikzpicture}